\documentclass[11pt]{amsart}

\usepackage{amsmath}              
\usepackage{amsfonts}              
\usepackage{amsthm} 
\usepackage[pdftex]{graphicx}
\usepackage{mdwlist}
\usepackage{amssymb}
\usepackage{multicol}
\usepackage{pifont}
\usepackage{wrapfig}

\newtheorem{thm}{Theorem}
\newtheorem{lem}[thm]{Lemma}
\newtheorem{prop}[thm]{Proposition}

\theoremstyle{definition}

\theoremstyle{definition}

\theoremstyle{plain}

\allowdisplaybreaks[3]  

\begin{document}

\title[Symbolic Generic Initial System of Points on a Conic]{The Symbolic Generic Initial System of Points on an Irreducible Conic}
\author{Sarah Mayes}

\maketitle
\vspace*{-2em} 

\begin{abstract}
In this note we study the limiting behaviour of the symbolic generic initial system $\{ \text{gin}(I^{(m)}) \}$ of an ideal $I \subseteq K[x,y,z]$ corresponding to an arrangement of $r$ points of $\mathbb{P}^2$ lying on an irreducible conic.  In particular, we show that the \textit{limiting shape} of this system is the subset of $\mathbb{R}^2_{\geq 0}$ such consisting of all points above the line through $(\text{min} \{ \frac{r}{2}, 2 \},0)$ and $(0, \text{max} \{ \frac{r}{2}, 2 \})$.
\end{abstract}

The general research trend looking at the asymptotic behaviour of collections of algebraic objects is motivated by the idea that there is often a structure revealed in the limit that is difficult to see when studying individual objects (see, for example, \cite{ELS01}, \cite{Siu01}, \cite{Huneke92}, and \cite{ES09}).  The asymptotic behaviour of a collection of monomial ideals $\mathrm{a}_{\bullet}$ such that $\mathrm{a}_i \cdot \mathrm{a}_j \subseteq \mathrm{a}_{i+j}$ (a \textit{graded system of monomial ideals}) can be described by its \textit{limiting shape} $P$.  If $P_{\mathrm{a}_i}$ denotes the Newton polytope of $\mathrm{a}_i$, then the limiting shape $P$ is defined to be the limit $\lim_{m \rightarrow \infty} \frac{1}{m} P_{\mathrm{a}_m}$ (\cite{Mayes12a}).  In addition to giving a simple geometric interpretation of the limiting behaviour, $P$ completely determines the asymptotic multiplier ideals of $\mathrm{a}_{\bullet}$ (see \cite{Howald01}).

Generic initial ideals have a nice combinatorial structure, but are often difficult to compute and usually have complicated sets of generators (see \cite{Green98} for a survey or \cite{Cimpoeas06} and  \cite{Mayes12c} for examples).  This motivates a series of work describing the limiting shape of generic initial systems,  $\{ \text{gin}(I^m) \}_m$, and of symbolic generic initial systems, $\{ \text{gin}(I^{(m)}) \}_m$ (\cite{Mayes12a}, \cite{Mayes12b}, \cite{Mayes12c}, \cite{Mayes12e}).  The goal of this paper is to describe the limiting shape of the symbolic generic initial system of the ideal of $r$ points in $\mathbb{P}^2$ lying on an irreducible conic.

We will see that when $I$ is the ideal of points in $\mathbb{P}^2$, each of the polytopes $P_{\text{gin}(I^{(m)})}$, and thus $P$ itself, can be thought of as a subset of $\mathbb{R}^2$.  The following theorem describes $P$ in the case we are interested in. 


\begin{thm}
\label{thm:mainthm}
Let $I \subseteq R = K[x,y,z]$ be the ideal of $r >1$ distinct points $p_1, \dots, p_r$ of $\mathbb{P}^2$ lying on an irreducible conic and let $P \subseteq \mathbb{R}^2_{\geq 0}$ be the limiting shape of the reverse lexicographic symbolic generic initial system $\{ \textnormal{gin}(I^{(m)})\}_m$.  If $ r \geq 4$, then $P$ has a boundary defined by the line through the points $(2,0)$ and $(0, \frac{r}{2})$ (see Figure \ref{fig:limitingshape}).  If $r =2$ or $r=3$, then $P$ has a boundary defined by the line through the points $(\frac{r}{2}, 0)$ and $(0, 2)$.   
\end{thm}

\begin{figure}
\begin{center}
\includegraphics[width=4cm]{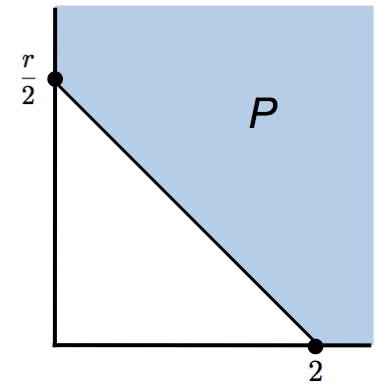}
\end{center}
\caption{The limiting shape $P$ of $\{ \text{gin}(I^{(m)})\}_m$ where $I$ is the ideal of $r \geq 4$ points lying on an irreducible conic.}
\label{fig:limitingshape}
\end{figure}

The proof of this theorem is an application of ideas that have been described elsewhere. Rather than repeating arguments here, we refer the reader elsewhere for details where necessary.

The following result describes the structure of the individual ideals $\text{gin}(I^{(m)})$ that make up the generic initial system.

\begin{thm}
\label{thm:formofgin}
Suppose that $I \subseteq K[x,y,z]$ be the ideal of a set of distinct points of $\mathbb{P}^2$.  Then the minimal generators of $\textnormal{gin}(I^{(m)})$ are 
$$\big\{ x^{\alpha(m)}, x^{\alpha(m)-1}y^{\lambda_{\alpha(m)-1}(m)}, \dots, xy^{\lambda_1(m)}, y^{\lambda_0(m)}\big \}$$
for some positive integers $\lambda_0(m), \dots, \lambda_{\alpha(m)-1}$ such that $\lambda_0(m)> \lambda_1(m) > \cdots >\lambda_{\alpha(m)-1}(m)$.  Further, if the minimal free resolution of $I^{(m)}$ is of the form
$$0 \rightarrow F_1=\bigoplus_{i=1}^{\psi} R(-u_i) \rightarrow F_1=\bigoplus_{i=1}^{\mu} R(-d_i) \rightarrow I^{(m)} \rightarrow 0$$
with $U(m)= \textnormal{max}\{ u_i\}$ and $D(m) = \textnormal{min} \{ d_i\}$, then 
$$\alpha(m) = D(m)$$
and 
$$\lambda_0(m) = U(m)-1.$$
\end{thm}

\begin{proof}
The first part of the theorem is Corollary 2.9 of \cite{Mayes12c} and follows from results in \cite{BS87} and \cite{HS98}.  The second statement follows the a result of Hilbert-Burch, which says that, with the notation in the theorem, the minimal free resolution of $\text{gin}(I^{(m)})$ is of the form 
$$0  \rightarrow G_1  \rightarrow G_0   \rightarrow \text{gin}(I^{(m)}) \rightarrow 0$$
where $G_1 =  \bigoplus_{i=0}^{\alpha(m)} R(-\lambda_i(m)-i-1)$ and $G_0 = \big[\bigoplus_{i=0}^{\alpha(m)} R(-\lambda_i(m)-i)\big] \oplus R(-\alpha(m))$ (Corollary 4.15 of \cite{Green98}).  A \textit{consecutive cancellation} takes a sequence $\{ \beta_{i,j} \}$ to a new sequence by replacing $\beta_{i,j}$ by $\beta_{i,j}-1$ and $\beta_{i+1,j}$ by $\beta_{i+1,j}-1$.  The `Cancellation Principle' says that the graded Betti numbers of $J$ can be obtained by the graded Betti numbers of $\text{gin}(J)$ by making a series of consecutive cancellations (see Corollary 1.21 of \cite{Green98}).
Since $\lambda_0(m) +1 > \lambda_i(m)+i$ for all $i$, $\beta_{1,\lambda_0+1}\geq 1$ does not change with any such consecutive cancellation; thus, $R(-\lambda_0(m)-1)$ is the summand of $F_1$ with the largest shift.  
Likewise, $\alpha(m)<\lambda_i(m)+i+1$ for all $i$ so $\beta_{0, \alpha(m)} \geq 1$ does not change with a consecutive cancellation; thus, $R(-\alpha(m))$ is the summand of $F_0$ with the smallest shift.
\end{proof}

In the case where $m$ is even and $I$ is the ideal of $r\geq 3$ points lying on an irreducible conic in $\mathbb{P}^2$, we will see in Proposition \ref{prop:resolutions} that we can write down the entire minimal free resolution of $I^{(m)}$.  This will give us $D(m)$ and $U(m)$ when $m$ is even so that we can find the powers of $x$ and $y$, $x^{\alpha(m)} = x^{D(m)}$ and $y^{\lambda_0(m)}=y^{U(m)-1}$, that appear in a minimal generating set of $\text{gin}(I^{(m)})$.   In particular, Proposition \ref{prop:resolutions} implies the following. 

\begin{lem}
\label{lem:highlowshifts}
Suppose that $I$ is the ideal of $r \geq 3$ points in $\mathbb{P}^2$ lying on an irreducible conic and use the notation of the previous theorem.
\begin{itemize}
\item[(a)] If $r \geq 4$ is even, $D(m)=2m$ and $U(m) = \frac{rm}{2}+2$.
\item[(b)]  If $r>4$ is odd and $m$ is even, then $D(m) = 2m$ and $U(m) = \frac{rm}{2}+2$.
\item[(c)]  If $r = 3$ and $m$ is even, then $D(m) =\frac{3m}{2}$ and $U(m) =2m+1$.
\end{itemize}
\end{lem}

 By Lemma \ref{thm:formofgin}, each of the generic initial ideals $\text{gin}(I^{(m)})$ is generated in the variables of $x$ and $y$, so we can think of each Newton polytope $P_{\text{gin}(I^{(m)})}$, and thus the limiting shape $P$ itself, as a subset of $\mathbb{R}^2$.

The following result is the key for proving the main theorem:  it describes when the limiting polytope $P$ of the symbolic generic initial system in $\mathbb{P}^2$ is defined by a single boundary line. The proof is contained in \cite{Mayes12c}.

\begin{prop}[Corollary 2.16 of \cite{Mayes12c}]
\label{prop:maxvolumeattained}
Let $I \subseteq K[x,y,z]$ be the ideal of $r$ distinct points in $\mathbb{P}^2$ and let $P$ be the limiting shape of the symbolic generic initial system $\{ \textnormal{gin}(I^{(m)})\}_m$. Suppose that the $x$-intercept $\gamma_1$ and the $y$-intercept $\gamma_2$ of the boundary of $P$ are such that $\gamma_1 \cdot \gamma_2 = r$.  Then the limiting polytope $P$ has a boundary defined by the line passing through $(\gamma_1, 0)$ and $(0, \gamma_2)$.
\end{prop}

With these results in mind, we can now prove the main theorem.

\begin{proof}[Proof of Theorem \ref{thm:mainthm}]

Suppose first that $r \geq 4$ and that $m$ is even if $r$ is odd.  By Theorem \ref{thm:formofgin} and Lemma \ref{lem:highlowshifts}, $x^{D(m)} = x^{2m}$ and $y^{U(m)-1} = y^{\frac{rm}{2}+1}$ are the smallest powers of $x$ and $y$ contained in $\text{gin}(I^{(m)})$.  This means that the intercepts of the boundary of $P_{\text{gin}(I^{(m)})}$ are $(2m,0)$ and $(0, \frac{rm}{2}+1)$.  Thus, the intercepts of the boundary of the limiting polytope $P$ of the entire symbolic generic initial system are $(\lim_{m \rightarrow \infty} \frac{2m}{m},0) = (2, 0)$ and $(0, \lim_{m \rightarrow \infty} \frac{rm/2}{m}+1) = (0,\frac{r}{2})$.\footnote{In the case where $r$ is odd we take the limits over even $m$.}  By Proposition \ref{prop:maxvolumeattained}, the fact that $\frac{r}{2} \cdot 2 = r$ implies that  the limiting polytope $P$ is as claimed.

Now suppose that $r=3$ and that $m$ is even.  By the same argument as above, the intercepts of the boundary of the limiting polytope $P$ are $(\lim_{m \rightarrow \infty} \frac{3m/2}{m},0) = (\frac{3}{2},0)$ and $(0,\lim_{m\rightarrow \infty} \frac{2m}{m}) = (0,2)$.  Since $\frac{3}{2} \cdot 2 = 3$, the limiting polytope is as claimed.

The case where $r=2$ follows from the main theorem of \cite{Mayes12a} since $I$ is a type (1,2) complete intersection in this case.
\end{proof}

It remains to prove Lemma \ref{lem:highlowshifts} which follows immediately from the next proposition.  In particular, we will write the minimal free resolutions of the ideals $I^{(m)}$ when $I$ is the ideal of $r \geq 3$ points on an irreducible conic and $m$ is even if $r$ is odd.  

\begin{prop}
\label{prop:resolutions}
Let $I$ be the ideal of $r \geq 3$ points of $\mathbb{P}^2$ lying on an irreducible conic and suppose that that the minimal free resolution of $I^{(m)}$ is of the form
$$0 \rightarrow G_1 \rightarrow G_0 \rightarrow I^{(m)} \rightarrow 0.$$
\begin{itemize}
\item[(a)]  If $r$ is even,
$$G_0 = \bigoplus_{j=0}^m R\big(-2(m-j)-\frac{rj}{2}\big)$$
and 
$$G_1 = \bigoplus_{j=1}^m R\big(-2(m-j) - \frac{rj}{2}-2\big).$$
\item[(b)]  If $r\geq 5$ is odd and $m$ is even,
$$G_0 = \Bigg[ \bigoplus_{j=0}^{m/2} R(-2m-j(r-4)) \Bigg] \oplus \Bigg[ \bigoplus_{j=0}^{m/2-1} R^2\big(-2m-j(r-4)-\frac{r-1}{2}+1\big)\Bigg]$$
and 
$$G_1 = \Bigg[ \bigoplus_{j=1}^{m/2} R(-2m-j(r-4)-2) \Bigg] \oplus \Bigg[ \bigoplus_{j=0}^{m/2-1} R^2\big(-2m-j(r-4)-\frac{r-1}{2}\big) \Bigg].$$
\item[(c)] If $r=3$ and $m$ is even,
$$G_0 = R\Big(-\frac{3m}{2}\Big) \oplus \Bigg[ \bigoplus_{j=0}^{m/2-1} R^3\big(-\frac{3m}{2} - j-1\big) \Bigg]$$
and 
$$G_1 = \bigoplus_{j=0}^{m/2-1} R^3\Big(-\frac{3m}{2} -j-2\Big). $$ 
\end{itemize}
\end{prop}

To prove this proposition we will follow the results of Catalisano described in \cite{Catalisano91} that can be used to compute the minimal free resolution of \textit{any} fat point ideal
$$I_{(m_1, \dots, m_r)} = I^{m_1}_{p_1} \cap I^{m_2}_{p_2} \cap \cdots \cap I^{m_r}_{p_r}$$
as long as the points $p_1, \dots, p_r$ lie on an irreducible conic.  The following is a specialization of Catalisano's work to the case where $r\geq4$ and $m_i = m$ for all $i$ (that is, when $I$ is the ideal of a uniform fat point subscheme).

\begin{prop}[\cite{Catalisano91}]
\label{prop:algorithm}
Let $I$ be the ideal of $r \geq 4$ points of $\mathbb{P}^2$ lying on an irreducible conic.  Suppose that the minimal free resolution of $I^{(t-1)}$ is of the form 
$$0 \rightarrow F'_1 \rightarrow F'_0 \rightarrow I^{(t-1)} \rightarrow 0$$
where $F'_1 = \oplus_{i=1}^{\mu-1} R(-u_i)$ and $F'_0 = \oplus_{i=1}^{\mu} R(-d_i)$ and that the minimal free resolution of $I^{(t)}$ is of the form
$$0 \rightarrow F_1 \rightarrow F_0 \rightarrow I^{(t)} \rightarrow 0.$$
If $rt$ is even
\begin{itemize}
\item[(1)] $F_1 = [\oplus_{i=1}^{\mu-1} R(-u_i-2)] \oplus R(-\frac{rt}{2} - 2)$ and \newline $F_0 = [\oplus_{i=1}^{\mu} R(-d_i-2)] \oplus R(-\frac{rt}{2})$
\end{itemize}
while if $rt$ is odd
\begin{itemize}
\item[(2)] $F_1 = [\oplus_{i=1}^{\mu-1} R(-u_i-2)] \oplus R^2(-\frac{rt+1}{2} - 1)$ and \newline $F_0 = [\oplus_{i=1}^{\mu} R(-d_i-2)] \oplus R^2(-\frac{rt+1}{2})$.
\end{itemize}
Therefore,  one can apply this result $m$ times to find the minimal free resolution of $I^{(m)}$.   That is, first find the minimal free resolution of $I^{(1)}=I$ from $0 \rightarrow 0 \rightarrow R(0) \rightarrow I^{(0)} \rightarrow 0$, then find the minimal free resolution of $I^{(2)}$ from that of $I$, and so on.
\end{prop}

\begin{proof}[Sketch of Proof of Proposition \ref{prop:resolutions}]
We will give an idea of how to find the minimal free resolutions of the ideals $I^{(m)}$ using the algorithm in Proposition \ref{prop:algorithm}.

If we are in case (a) where $r$ is even, $rt$ is even for all $t$, so to find the resolution of $I^{(t)}$ from the minimal free resolution of $I^{(t-1)}$ we follow the first case of Proposition \ref{prop:algorithm}.  In particular, to find the resolution of $I^{(m)}$ from the resolution of $I^{(0)}$, we apply part (1) exactly $m$ times for $t = 1, \dots, m$. 

If we are in case (b) where $r$ is odd, $rt$ is odd for odd $t$ and $rt$ is even for even $t$.  Thus, we need to apply both cases of Proposition \ref{prop:algorithm} to find the resolution of $I^{(m)}$.  To obtain the resolution 
$$0 \rightarrow F_1 \rightarrow F_0 \rightarrow I^{(t)} \rightarrow 0$$ 
of $I^{(t)}$ from the resolution 
$$0 \rightarrow F''_1 \rightarrow F''_0 \rightarrow I^{(t-2)} \rightarrow 0$$
of $I^{(t-2)}$ when $t$ is even, one needs to:
\begin{itemize}
\item shift each summand of $F_0$ and $F_1$ by $-4$; 
\item add $R^2(-\frac{r(t-1)+1}{2}-3)$ and $R(-\frac{rt}{2}-2)$ to $F''_1$; and
\item add $R^2(-\frac{r(t-1)+1}{2}-3)$ and $R(-\frac{rt}{2})$ to $F''_0$.
\end{itemize}
If $m$ is even, we can follow this procedure $\frac{m}{2}$ times with $t = 2, 4, 6, \dots, m$ to find the resolution of $I^{(m)}$ from that of $I^{(0)}$.

For case (c) when $r=3$ and $m$ is even, one needs to use other results of \cite{Catalisano91} beyond those stated in Proposition \ref{prop:algorithm}.  The general idea is the same as above:  use a sequence of fat point schemes $Z_0 = m(p_1+p_2+p_3), Z_1, Z_2, \dots, Z_H = 0(p_1+p_2+p_3)$ and find the minimal free resolution of $I_{Z_{H-1}}$ from that of $I_{Z_H}$, then find the minimal free resolution of $I_{Z_{H-2}}$ from that of $I_{Z_{H_1}}$, and so on, until we can find the minimal free resolution of $I^{(m)} = I_{Z_0}$ from that of $I_{Z_1}$.  However, when $r=3$ not all of the $Z_i$ will be uniform fat point subschemes.  In particular,  subsequences of the form $Z_l = t(p_1+p_2+p_3)$, $Z_{l+1} = (t-1)p_1+(t-1)p_2+tp_3$, $Z_{l+2} = (t-1)p_1+(t-2)p_2+(t-1)p_3$, $Z_{l+3} = (t-2)(p_1+p_2+p_3)$ come together to form the sequence $Z_0, \dots, Z_H$.  See \cite{Catalisano91} for further details.
\end{proof}

\bibliography{ConicBib}
\bibliographystyle{amsalpha}
\nocite{*}

\end{document}